\documentclass{amsart}

\usepackage{a4wide,latexsym,amssymb,amsthm,amsmath,color}

\theoremstyle{plain}

\newtheorem{theorem}{Theorem}
\newtheorem{lemma}{Lemma}
\newtheorem{proposition}{Proposition}

\theoremstyle{remark}

\newtheorem{remark}{Remark}

\def\tg{\mathrm{tg}}

\def\R{\mathbb{R}}

\def\H{\mathbb{H}}

\newcommand{\nks}{\ensuremath{\mathbb{S}^3 \times
\mathbb{S}^3}}

\title[Geodesic lines on nearly  K\"ahler $\nks$]
{Geodesic lines on nearly  K\"ahler $\nks$}

\author[Djori\'c]{Milo\v s Djori\'c}
\address{University of Belgrade, Faculty of Mathematics, Studentski trg 16, pb. 550, 11000 Belgrade,
Serbia} \email{milosdj@matf.bg.ac.rs}

\author[Djori\'c]{Mirjana  Djori\'c}
\address{University of Belgrade, Faculty of Mathematics, Studentski trg 16, pb. 550, 11000 Belgrade,
Serbia} \email{mdjoric@matf.bg.ac.rs}

\author[Moruz]{Marilena Moruz}
\address{KU Leuven, Department of Mathematics, Celestijnenlaan 200B, 3001 Leuven, Belgium}
\email{marilena.moruz@kuleuven.be}

\thanks{The first two authors are  partially
supported by Ministry of Education, Science and Technological Development, Republic of Serbia, project 174012.}
\thanks{The third author is a postdoctoral fellow of the Research Foundation Flanders (FWO)}

\numberwithin{equation}{section}
\numberwithin{lemma}{section}\numberwithin{example}{section}\numberwithin{theorem}{section}\numberwithin{proposition}{section}

\begin{document}



\begin{abstract}
A Nearly  K\"ahler manifold is an almost Hermitian manifold with the weakened  K\"ahler condition, that is, instead of being zero, the covariant derivative of the almost complex structure is skew-symmetric.	 We give the explicit parametrization of geodesic lines on nearly  K\"ahler $\nks$.
\end{abstract}

\maketitle

\section{Introduction}

A nearly  K\"ahler manifold  is an almost Hermitian manifold  $(M,g,J)$ with the property that the tensor field $\nabla J$ is skew-symmetric:	
$(\nabla_X J) Y + (\nabla_Y J) X  = 0,$ for all $\, X,Y \in TM,$ where $\nabla$ is the Levi-Civita connection of the metric $g$.
The first example of such manifolds was introduced on $\mathbb{S}^6$ by Fukami and Ishihara in \cite{fukami}   and, later, these manifolds have been intensively studied by  A. Gray in \cite{gray}, who generalized the classical holonomy concept by
introducing a classification principle for non-integrable special Riemannian geometries
and studied the defining differential equations of each class.
The structure theorem of Nagy \cite{nagy, nagy2} asserts that a strict and complete nearly K\"ahler manifold (of arbitrary dimension) writes as a Riemannian product of homogeneous nearly K\"ahler spaces, twistor spaces over quaternionic K\"ahler manifolds and $6$-dimensional nearly K\"ahler manifolds.
Moreover, Butruille  has  shown in \cite{but} that the only nearly K\"ahler homogeneous manifolds of dimension $6$ are the compact spaces $\mathbb{S}^6$, $\nks$, $\mathbb{C}P^3$ and the flag manifold  $SU(3)/U(1)\times U(1)$ (where the last three are not endowed with the standard metric).
Furthermore, Foscolo et Haskins  found in \cite{Haskins} the first two complete non-homogeneous nearly K\"ahler structures  on $\mathbb{S}^6$ and on $\nks$. This way they addressed an important problem in the field, namely the absence of any complete non-homogeneous examples.\\

More recently interest in nearly K\"ahler manifolds increased because these manifolds are examples of geometries with torsion and therefore they have applications in mathematical physics \cite{Agricola}.
Moreover, $6$-dimensional nearly K\"ahler manifolds are Einstein and are related to the existence of a Killing spinor \cite{Fridrich}, which inspires their further investigation.
In this paper we continue the research of the nearly K\"ahler manifold $\nks$ (see the previous work in \cite{Burcu}, \cite{dioosvranckenwang}, \cite{ZH}, for instance) by studying its geodesics.
In Section 2 we recall the basic properties of the nearly K\"ahler $\nks$ and in Section 3 our main results are stated and proved. Theorem 3.1. presents the parametrization of geodesic lines on $\nks$ and Proposition 3.1. presents their features.

\section{Preliminaries}\label{prelim}

In this section we recall  the homogeneous nearly K\"ahler structure of  $\mathbb{S}^3\times\mathbb{S}^3$. For more details we refer the reader to \cite{complexsurfaces}. By the natural identification $T_{(p,q)}(\mathbb{S}^3\times\mathbb{S}^3)\cong T_p\mathbb{S}^3\oplus T_q\mathbb{S}^3$, we write a tangent vector at $(p,q)$ in $\mathbb{S}^3\times\mathbb{S}^3$ as $Z(p,q)=(U(p,q),V(p,q))$ or simply $Z=(U,V)$. Regarding the $3$-sphere in $\mathbb{R}^4$ as the set of all unit quaternions in $\mathbb{H}$ and using the notations $i,j,k$ to denote the imaginary units of $\mathbb{H}$,
the vector fields defined by
\begin{equation} \label{baza}
\begin{array}{ll}
\tilde E_1(p,q)=(pi,0),&\tilde  F_1(p,q)=(0,qi),\\
\tilde E_2(p,q)=(pj,0),& \tilde  F_2(p,q)=(0,qj),\\
\tilde E_3(p,q)=-(pk,0),&\tilde   F_3(p,q)=-(0,qk),
 \end{array} \end{equation}
are mutually orthogonal with respect to the usual Euclidean product metric on  $\mathbb{S}^3\times\mathbb{S}^3$. The Lie brackets are $[\tilde E_i,\tilde E_j]=-2\varepsilon_{ijk}\tilde E_k,$ $[\tilde F_i,\tilde F_j]=-2\varepsilon_{ijk}\tilde F_k$ and $[\tilde E_i,\tilde F_j]=0$, where
\begin{align*}
\varepsilon_{ijk}=\left\{
\begin{array}{l}
1, \quad \text{if} \ (ijk)\ \text{is an even permutation of } (123),\\
-1, \quad \text{if} \ (ijk)\ \text{is an odd permutation of } (123),\\
0, \quad \text{otherwise.}
\end{array}
\right.
\end{align*}
The almost complex structure $J$ on the nearly K\"ahler $\mathbb{S}^3\times\mathbb{S}^3$ is defined by
\begin{equation}\label{defJ}
J(U,V)_{(p,q)}=\frac{1}{\sqrt{3}} (2pq^{-1}V-U, -2qp^{-1}U+V )_{(p,q)},
\end{equation}
for $(U,V)\in T_{(p,q)}(\mathbb{S}^3\times\mathbb{S}^3)$
and the Hermitian metric associated with the usual Euclidean product metric on $\mathbb{S}^3\times\mathbb{S}^3$ on $\mathbb{S}^3\times\mathbb{S}^3$ is given by
\begin{align}\label{g}
g(Z,Z')=&\frac{1}{2}(\langle Z,Z' \rangle+ \langle JZ,JZ' \rangle)\\
	=&\frac{4}{3}(\langle U,U' \rangle+\langle V,V'\rangle)-\frac{2}{3}(\langle p^{-1}U,q^{-1}V' \rangle+\langle p^{-1}U',q^{-1}V \rangle),\nonumber
\end{align}
where $Z = (U, V )$, $Z' = (U', V')$, in the first line $\langle \cdot,\cdot \rangle $ stands for the usual Euclidean product metric on $\mathbb{S}^3\times\mathbb{S}^3$, while in the second line $\langle \cdot,\cdot \rangle $ stands for the usual Euclidean metric on $\mathbb{S}^3$. By definition, the almost complex structure $J$ is compatible with the metric $g$.\\
Using the Koszul formula (see Lemma 2.1. in \cite{complexsurfaces} for more details), one finds
that the Levi-Civita connection $\tilde\nabla$ on $\mathbb{S}^3\times \mathbb{S}^3$ with respect to the metric $g$ is given by
\[ \begin{array}{ll}
\tilde\nabla_{\tilde E_i}\tilde E_j=-\varepsilon_{ijk}\tilde E_k & \tilde\nabla_{\tilde E_i}\tilde F_j=\frac{\varepsilon_{ijk}}{3}(\tilde E_k-\tilde F_k) \\
\tilde\nabla_{\tilde F_i}\tilde E_j=\frac{\varepsilon_{ijk}}{3}(\tilde F_k -\tilde E_k)& \tilde\nabla_{\tilde F_i}\tilde F_j=-\varepsilon_{ijk}\tilde F_k.
 \end{array} \]
\noindent
Consequently, computing
\begin{equation}\label{nablaJ}
 \begin{array}{ll}
(\tilde\nabla_{\tilde E_i}J)\tilde E_j=-\frac{2}{3\sqrt{3}}\varepsilon_{ijk}(\tilde E_k+2\tilde F_k), &    (\tilde\nabla_{\tilde E_i}J)\tilde F_j=-\frac{2}{3\sqrt{3}}\varepsilon_{ijk} (\tilde E_k-\tilde F_k),\\
(\tilde\nabla_{\tilde F_i}J)\tilde  E_j=-\frac{2}{3\sqrt{3}} \varepsilon_{ijk} (\tilde E_k-\tilde F_k), & (\tilde\nabla_{\tilde F_i}J)\tilde F_j=-\frac{2}{3\sqrt{3}}\varepsilon_{ijk}(2\tilde E_k+\tilde F_k),
 \end{array}
\end{equation}
we conclude that the $(1, 2)$-tensor field $G=\nabla J$ is skew-symmetric and therefore $\mathbb{S}^3\times\mathbb{S}^3$, equipped with $g$ and $J$, becomes a nearly K\"ahler manifold.
Moreover $G$ satisfies
\begin{equation}
G(X,JY)=-JG(X,Y), \quad g(G(X,Y),Z)+g(G(X,Z),Y)=0,
\end{equation}
for any vector fields $X,Y,Z$ tangent to $\mathbb{S}^3\times\mathbb{S}^3$.

The almost product structure $P$ introduced in \cite{complexsurfaces} is defined as
\begin{equation}\label{defP}
PZ=(pq^{-1}V,qp^{-1}U), \quad Z=(U,V)\in T_{(p,q)}(\mathbb{S}^3\times\mathbb{S}^3)
\end{equation}
and it has the following properties:
\begin{align*}
&P^2=Id \quad \text{($P$ is involutive)},\\
&PJ=-JP \quad \text{($P$ and $J$ anti-commute)},\\
&g(PZ,PZ')=g(Z,Z') \quad \text{($P$ is compatible with $g$)},\\
&g(PZ,Z')=g(Z,PZ') \quad \text{($P$ is symmetric)}.
\end{align*}
Moreover, the almost product structure $P$ can be expressed in terms of the usual product structure $QZ=Q(U,V)=(-U,V)$ and vice versa:
\begin{equation*}
QZ = \frac{1}{\sqrt{3}} (2 PJZ - JZ),\quad
PZ = \frac{1}{2}(Z-\sqrt{3}QJZ).
\end{equation*}
Next, we recall the relation between the Levi-Civita connections $\tilde\nabla$ of $g$ and $\nabla^E$ of the Euclidean product metric $\langle\cdot,\cdot\rangle$.
\begin{lemma}\label{nablaE}\cite{dioosvranckenwang}
The relation between the nearly K\"ahler connection $\tilde\nabla$ and the Euclidean connection $\nabla^E$ is
\begin{equation}\label{relnablas}
\nabla^E_XY=\tilde\nabla_XY+\frac{1}{2}(JG(X,PY)+JG(Y,PX)).
\end{equation}
\end{lemma}
We also recall here a useful formula from (\cite{dioosvranckenwang}), decomposing $D_XY$ along the tangent and the normal directions:
\begin{equation}\label{Euclconn}
D_XY=\nabla^E_XY+\frac{1}{2}\langle D_XY,(p,q) \rangle (p,q)+\frac{1}{2}\langle D_XY,(-p,q) \rangle (-p,q),
\end{equation}
where $D$ is the Euclidean connection on $\R^8$ and $X,Y$ are tangent vector fields on $\nks$.


\section{Geodesic lines on \nks }\label{geod lines}

In order to study and classify certain types of submanifolds of \nks, it is interesting and useful to know how its geodesic lines look like.
Since for unitary quaternions $a, b$ and $c$, the map $f:\nks\rightarrow \nks$ given by
$(p, q)\mapsto (apc^{-1}, bqc^{-1})$ preserves  the usual metric $\langle \cdot ,\cdot \rangle $ and the almost complex structure $J$, it is an isometry of $(\nks, g)$ (cf. \cite{dioosvranckenwang} and remark after Lemma 2.2 in \cite{PodestaSpiro}).
Therefore, it is enough to obtain the parametrization of geodesic lines through the point $(1,1)$, as the geodesic lines through the point $(a,b)$ are given by $\tilde\gamma(t)=(ax(t),by(t))$. We prove the following.

\begin{theorem}\label{Geodlinesth}
The geodesic lines on \nks through the point $(1,1)$ have the following parametrization:
\begin{enumerate}
\item $\gamma(t)=(\cos(at)+\sin(at)i,\cos(at)-\sin(at)i)$, $a\in \R\setminus\{0\}$;
\item $\gamma(t)=(\cos(at)+\sin(at)i,\cos(\widetilde{a}t)+\sin(\widetilde{a}t)i),$ where $c_1\in Im\H\setminus\{0\}$, $d_1\in\mathbb{R}$, $a=\frac{1+d_1}{2}\| c_1\|$, $\widetilde{a}=\frac{1-d_1}{2}\| c_1\|$;
\item \begin{eqnarray*}\label{eq15}
\gamma(t)=\bigg(\left(\frac{1}{1+\varphi^2}\cos(At)+\frac{\varphi^2}{1+\varphi^2}\cos(Bt)\right)+
\left(\frac{1}{1+\varphi^2}\sin(At)+\frac{\varphi^2}{1+\varphi^2}\sin(Bt)\right)i+\\ \nonumber
\left(\frac{\varphi}{1+\varphi^2}\sin(At)-\frac{\varphi}{1+\varphi^2}\sin(Bt)\right)j-
\left(-\frac{\varphi}{1+\varphi^2}\cos(At)+\frac{\varphi}{1+\varphi^2}\cos(Bt)\right)k\,,\\ \nonumber
\left(\frac{1}{1+\widetilde{\varphi}^2}\cos(\widetilde{A}t)+\frac{\widetilde{\varphi}^2}{1+\widetilde{\varphi}^2}\cos(\widetilde{B}t)\right)+
\left(\frac{1}{1+\widetilde{\varphi}^2}\sin(\widetilde{A}t)+
\frac{\widetilde{\varphi}^2}{1+\widetilde{\varphi}^2}\sin(\widetilde{B}t)\right)i+\\ \nonumber
\left(\frac{\widetilde{\varphi}}{1+\widetilde{\varphi}^2}\sin(\widetilde{A}t)-
\frac{\widetilde{\varphi}}{1+\widetilde{\varphi}^2}\sin(\widetilde{B}t)\right)j-
\left(-\frac{\widetilde{\varphi}}{1+\widetilde{\varphi}^2}\cos(\widetilde{A}t)+
\frac{\widetilde{\varphi}}{1+\widetilde{\varphi}^2}\cos(\widetilde{B}t)\right)k\bigg),\nonumber
\end{eqnarray*}
\end{enumerate}
where $c_1,c_2\in Im\H\setminus\{0\}$,  $d_1\in\mathbb{R}$, $a=\frac{1+d_1}{2}\| c_1\|$, $b=\frac{1}{2}\| c_2\|$, $c=\frac{2}{3}\|c_1\|$, $\tilde a=\frac{1-d_1}{2}\| c_1\|$, $\tilde b=-\frac{1}{2}\| c_2\|$, $\tilde c=c$, $$A=\frac{c+\sqrt{(2a-c)^2+4b^2}}{2},\qquad B=\frac{c-\sqrt{(2a-c)^2+4b^2}}{2},$$ $$\widetilde{A}=\frac{\widetilde{c}+\sqrt{(2\widetilde{a}-\widetilde{c})^2+4\widetilde{b}^2}}{2},\qquad \widetilde{B}=\frac{\widetilde{c}-\sqrt{(2\widetilde{a}-\widetilde{c})^2+4\widetilde{b}^2}}{2},$$ $$\varphi=\frac{c-2a+\sqrt{(c-2a)^2+4b^2}}{2b},\qquad \widetilde{\varphi}=\frac{\widetilde{c}-2\widetilde{a}+\sqrt{(\widetilde{c}-2\widetilde{a})^2+
4\widetilde{b}^2}}{2\widetilde{b}}.$$
\end{theorem}

\begin{remark}
With the substitution $\varphi=\tg\theta$, $\theta\in\mathbb{R}$, (analogously for $\widetilde{\theta}$) $\frac{1}{1+\varphi^2}=\cos^2\theta$, $\frac{\varphi^2}{1+\varphi^2}=\sin^2\theta$, $\frac{\varphi}{1+\varphi^2}=\sin\theta\cos\theta=\frac12\sin2\theta$, the parametrization in the case (3) becomes
\begin{eqnarray}\label{eq16}
\gamma(t)=\bigg((\cos^2\theta\cos(At)+\sin^2\theta\cos(Bt))+\nonumber
(\cos^2\theta\sin(At)+\sin^2\theta\sin(Bt))i+\\ \nonumber
\sin\theta\cos\theta(\sin(At)-\sin(Bt))j-
\sin\theta\cos\theta(-\cos(At)+\cos(Bt))k,\\ \nonumber
(\cos^2\widetilde{\theta}\cos(\widetilde{A}t)+\sin^2\widetilde{\theta}\cos(\widetilde{B}t))+
(\cos^2\widetilde{\theta}\sin(\widetilde{A}t)+\sin^2\widetilde{\theta}\sin(\widetilde{B}t))i+\\ \nonumber
\sin\widetilde{\theta}\cos\widetilde{\theta}(\sin(\widetilde{A}t)-\sin(\widetilde{B}t))j-
\sin\widetilde{\theta}\cos\widetilde{\theta}(-\cos(\widetilde{A}t)+\cos(\widetilde{B}t))k\bigg).\nonumber
\end{eqnarray}
\end{remark}

\begin{proof}
Our proof starts with the observation that  if  $\gamma(t)=(x(t),y(t))$ is a geodesic line on \nks, then
there exist functions $\alpha$, $\beta: \R \rightarrow Im\H$ such that
\begin{equation}\label{eq1a}
\gamma'(t)=(x'(t),y'(t))=(x(t)\alpha(t),y(t)\beta(t)).
\end{equation}
Namely, it is well-known that for  $v\in T_p\mathbb{S}^3$ there exists  $\tilde v\in Im\H$ such that $v=p\tilde v$.

For more convenience, we will sometimes omit the argument $t$ for the functions $x, y, \alpha, \beta,\gamma$ and their derivatives. From (\ref{eq1a}) it follows
\begin{equation}
\gamma''=(x\alpha\cdot \alpha +x\alpha',y\beta\cdot \beta +y\beta').
\end{equation}
Recall that for imaginary quaternions $\alpha_1,\alpha_2$ one has $\alpha_1\cdot \alpha_2=-<\alpha_1,\alpha_2>+\alpha_1\times\alpha_2$. Thus, $\alpha\cdot \alpha=-\|\alpha\|^2$ and we may write
\begin{align*}
\gamma''&=(x(\alpha'-\|\alpha\|^2),y(\beta'-\|\beta\|^2))\\
	&=(x\alpha',y\beta')-\| \alpha\|^2(x,0)-\| \beta\|^2(0,y).
\end{align*}
Identifying the tangent and normal parts in the previous relation, using (\ref{Euclconn}), it follows
\begin{equation}\label{x}
\nabla^E_{\gamma'}\gamma'=(x\alpha',y\beta')
\end{equation}
and Lemma \ref{nablaE} yields
\begin{equation}\label{a1}
\nabla^E_{\gamma'}\gamma'=\tilde \nabla_{\gamma'}\gamma'+JG(\gamma', P\gamma').
\end{equation}
An easy computation establishes the additional formula which allows us to evaluate $G$ for any tangent vector fields:
\begin{align}\label{formulaG}
G(X,Y)=\frac{2}{3\sqrt{3}}(p(\beta\times\gamma+\alpha\times\delta&+\alpha\times\gamma-2\beta\times\delta),
	q(-\alpha\times\delta-\beta\times\gamma+2\alpha\times\gamma-\beta\times\delta)),
\end{align}
for  $X=(p\alpha,q\beta),Y=(p\gamma,q\delta)\in T_{(p,q)}\nks$, $\alpha, \beta, \gamma, \delta \in
Im\H$.
As $\gamma'=(x\alpha,y\beta)$, $P\gamma'=(x\beta,y\alpha)$,  relation (\ref{formulaG}) implies
$G(\gamma', P\gamma')=\frac{2}{\sqrt 3}(x(\alpha\times\beta), y(\alpha\times\beta))$, which then gives
\begin{equation}\label{JG}
JG(\gamma', P\gamma')=\frac{2}{3}(x(\alpha\times\beta),-y(\alpha\times\beta)).
\end{equation}
Finally, as $\gamma$ is a geodesic, using \eqref{a1} and \eqref{JG}, we compute
\begin{equation}\label{eq2}
\alpha'(t)=\frac{2}{3}\alpha(t)\times \beta(t), \quad \beta'(t)=-\frac{2}{3}\alpha(t)\times \beta(t).
\end{equation}
From (\ref{eq2}) it follows $(\alpha(t)+\beta(t))' =0$ and therefore there exists $c_1\in Im\H$ such that
\begin{equation}\label{eq2a}
\alpha(t)+\beta(t) =c_1.
\end{equation}
Using (\ref{eq2}) and (\ref{eq2a}) it follows
\begin{equation}\label{eq2b}
(\alpha(t)-\beta(t))' =-\frac{2}{3} c_1\times (\alpha(t)-\beta(t)).
\end{equation}

If $c_1=0$, using (\ref{eq2a}) and (\ref{eq2b}), then $\alpha$ and $\beta$ are constants. Without loss of generality, we can assume that they are imaginary quaternions collinear with $i$:
\begin{equation}\label{c1}
\alpha=-\beta=ai,\qquad a\in\mathbb{R}\setminus\{0\}.
\end{equation}
\vskip.5cm
If $c_1\ne 0$, since $<\alpha(t)-\beta(t),c_1>=<\alpha(t)-\beta(t),\alpha(t)+\beta(t)>$ and $c_1$ is a constant vector, we have
 $<\alpha(t)-\beta(t),c_1>'= <-\frac{2}{3} c_1\times (\alpha(t)-\beta(t)),c_1>=0$, namely, $<\alpha(t)-\beta(t),c_1>=\tilde d_1\in \mathbb{R}$.
 Let us denote by $\varepsilon$ the part of $\alpha-\beta$ which is orthogonal to $c_1$:
\begin{equation}\label{epsilon}
\varepsilon(t)=\alpha(t)-\beta(t)-\frac{<\alpha(t)-\beta(t),c_1>}{\|c_1\|^2}c_1
\end{equation}
and compute
\begin{eqnarray}\label{de0}
\varepsilon'(t)&=-\frac{2}{3} c_1\times\varepsilon(t),\\
\label{de}
\varepsilon''(t)&=-\frac{4}{9}\|c_1\|^2 c_1\varepsilon(t).
\end{eqnarray}
Solving (\ref{de}) gives
\begin{equation}\label{de1}
\varepsilon(t)=\cos\left(\frac{2}{3}\|c_1\| t\right)c_2+\sin\left(\frac{2}{3}\|c_1\| t\right)c_3,
\end{equation}
for  $c_2,c_3\in Im\H$, $c_2,c_3\perp c_1$.
Having in mind (\ref{de0}), we conclude that
\begin{equation*}
\|c_1\| c_2=c_1\times c_3,\quad
\|c_1\| c_3=-c_1\times c_2.
\end{equation*}
Therefore
\begin{equation}\label{c3}
c_3=-\frac{c_1}{\|c_1\|}\times c_2
\end{equation}
and, consequently, $c_3\perp c_1,c_2$.

For $c_2=0$ it follows that $c_3=0$, $\varepsilon=0$, $\alpha-\beta= d_1 c_1$ ($d_1=\frac{\tilde d_1}{\|c_1\|^2}$), $\alpha+\beta=c_1$ and therefore
\begin{equation*}
\alpha=\frac{1+ d_1}{2}c_1,\quad
\beta=\frac{1- d_1}{2}c_1.
\end{equation*}
We can assume that $c_1$ is collinear with $i$ and get
\begin{equation}\label{c2}
\alpha=ai,\quad
\beta=\tilde ai,
\end{equation}
where $a=\frac{1+d_1}{2}\|c_1\|$, $\tilde a=\frac{1-d_1}{2}\|c_1\|$.

For $c_2\ne 0$, the vectors $\frac{c_1}{\|c_1\|}$, $\frac{c_2}{\|c_2\|}$, $-\frac{c_3}{\|c_3\|}$ form an orthonormal basis of $Im\mathbb{H}$. Since it is always possible to find a unit quaternion $h\in \mathbb{S}^3$ such that
\begin{equation*}
h^{-1}\frac{c_1}{\|c_1\|}h=i,\quad
h^{-1}\frac{c_2}{\|c_2\|}h=j,\quad
h^{-1}\frac{c_3}{\|c_3\|}h=-k,
\end{equation*}
we may assume to be working with the basis $\{i,j,-k\}$.
Using (\ref{epsilon}) and (\ref{de1}) we obtain
\begin{eqnarray}
&\alpha(t)+\beta(t)=c_1=\|c_1\|\,h\,i\,h^{-1},\\
&\alpha(t)-\beta(t)=d_1\,\|c_1\|\,h\,i\,h^{-1}+\|c_2\|\cos\left(\frac{2}{3}\|c_1\| t\right)h\,j\,h^{-1}-\|c_2\|\sin\left(\frac{2}{3}\|c_1\| t\right)h\,k\,h^{-1}.
\end{eqnarray}
Since $\|c_3\|=\|c_2\|$ from (\ref{c3}) and  $d_1=\frac{\tilde d_1}{\|c_1\|^2}\in \mathbb{R}$, it gives
\begin{eqnarray}\label{eq3}
\alpha(t)&=&\frac{1+d_1}{2}\| c_1\|h i h^{-1}+\frac{1}{2}\| c_2\|\cos\left(\frac{2}{3}\|c_1\| t\right) h j h^{-1}\nonumber\\
&-&\frac{1}{2}\| c_2\|\sin\left(\frac{2}{3}\|c_1\| t\right) h k h^{-1},\\
\beta(t)&=&\frac{1-d_1}{2}\| c_1\|h i h^{-1}-\frac{1}{2}\| c_2\|\cos\left(\frac{2}{3}\|c_1\| t\right) h j h^{-1}\nonumber\\
&+&\frac{1}{2}\| c_2\|\sin\left(\frac{2}{3}\|c_1\| t\right) h k h^{-1}, \nonumber
\end{eqnarray}
where $h\in \mathbb{S}^3$, $c_1,c_2\in Im\H\setminus\{0\}$,  $d_1\in\mathbb{R}$.

\begin{remark}
Notice that we had the freedom to make a rotation of the basis, in the following sense.
The map $(x,y)\to(hxh^{-1},hyh^{-1})$ is an isometry (rotation) of \nks and
 if $(1,1)\in \gamma=(x,y)$ then also $(1,1)\in \tilde\gamma=(\tilde x,\tilde y)=(hxh^{-1}, hyh^{-1})$. Moreover, $\tilde\gamma$ stays tangent to \nks since, for $\alpha$, $\beta: \R \rightarrow Im\H$ and using (\ref{eq1a}), we have
\begin{eqnarray}
\tilde\gamma'&=&(hx'h^{-1},hy'h^{-1})=(hx\alpha h^{-1},hy\beta h^{-1})\\\nonumber
&=&(hxh^{-1}h\alpha h^{-1},hyh^{-1}h\beta h^{-1})=(\tilde x\tilde\alpha,\tilde y\tilde\beta)\nonumber
\end{eqnarray}
for $\tilde\alpha=h\alpha h^{-1}$, $\tilde\beta=h\beta h^{-1}$, $\tilde\alpha$, $\tilde\beta: \R \rightarrow Im\H$.
Therefore, it is enough to solve
\begin{equation}\label{eq4}
(x'(t),y'(t)) =(x(t)\alpha(t), y(t)\beta(t))
\end{equation}
for
\begin{gather}\label{eq5}
\begin{array}{lll}
\alpha(t)&=&\frac{1+d_1}{2}\| c_1\| i +\frac{1}{2}\| c_2\|\cos\left(\frac{2}{3}\|c_1\| t\right) j -\frac{1}{2}\| c_2\|\sin\left(\frac{2}{3}\|c_1\| t\right)  k ,\\
\beta(t)&=&\frac{1-d_1}{2}\| c_1\| i -\frac{1}{2}\| c_2\|\cos\left(\frac{2}{3}\|c_1\| t\right) j +\frac{1}{2}\| c_2\|\sin\left(\frac{2}{3}\|c_1\| t\right) k.
\end{array}
\end{gather}
\end{remark}

In all the cases considered, we have to solve the system of differential equations
\begin{gather}\label{eq6}
\begin{array}{lll}
\alpha(t)&=&ai +b\cos (ct) j -b\sin(ct)k ,\\
\beta(t)&=&\tilde a i +\tilde b\cos(\tilde ct)j -\tilde b \sin (\tilde ct)k,
\end{array}
\end{gather}
which reduces to solving an equation of the form
\begin{equation}\label{eq7}
f'(t)=f(t)(ai+bj\cos (ct)-bk\sin(ct)),
\end{equation}
with the following constants:
\begin{enumerate}
\item if $c_1=0$, then $\tilde{a}=-a\neq0$, $b=\tilde{b}=c=\tilde{c}=0$;
\item if $c_1\neq 0$, $c_2=0$, then $a=\frac{1+d_1}{2}\|c_1\|$, $\tilde{a}=\frac{1-d_1}{2}\|c_1\|$, $b=\tilde{b}=c=\tilde{c}=0$;
\item if $c_1\neq 0$, $c_2\neq0$, then $a=\frac{1+d_1}{2}\| c_1\|$, $b=\frac{1}{2}\| c_2\|$, $c=\frac{2}{3}\|c_1\|$, $\tilde a=\frac{1-d_1}{2}\| c_1\|$, $\tilde b=-\frac{1}{2}\| c_2\|$, $\tilde c=c$.
\end{enumerate}

Cases (1) and (2) lead to the same differential equation of the form $f'(t)=f(t)\cdot ai$, which has an obvious solution which satisfies $f(0)=1$: $$f(t)=\cos(at)+i\sin(at).$$
The geodesics in these two cases have the parametrization
$$\gamma(t)=(\cos(at)+\sin(at)i,\cos(\tilde{a}t)+\sin(\tilde{a}t)i),$$ where in case (1) we put $\tilde{a}=-a\neq 0$ (so $a+\tilde a=0$), and in case (2) $a+\tilde{a}=\|c_1\|\neq0$.
\noindent
In the third case, we write explicitly
\begin{equation}\label{star}
f(t)=f_1(t)+if_2(t)+jf_3(t)+kf_4(t)=:g_1(t)+jg_2(t),\nonumber
\end{equation}
for $f_l$ real functions, where $l=\overline{1,4}$. Then the equation (\ref{eq7}) becomes
\begin{align*}
g_1'+jg_2'=(g_1+jg_2)(ai+bje^{ict})
\end{align*}
and we get the following system of differential equations
\begin{equation} \label{eq8}
\begin{array}{ll}
g_1'(t)=g_1ai-\overline{g_2}be^{ict} ,\\
g_2'(t)=g_2ai+\overline{g_1}be^{ict} .
 \end{array} \end{equation}
When we differentiate one of these equations and combine it with the other, we get the second order linear equations
\begin{equation} \label{eq9}
\begin{array}{ll}
g_1''-cig_1'+(a^2+b^2-ac)g_1=0 ,\\
g_2''-cig_2'+(a^2+b^2-ac)g_2=0.
 \end{array} \end{equation}
The characteristic equation is
\begin{equation}\label{eq10}
\lambda^2-ci\lambda+(a^2+b^2-ac)=0,
\end{equation}
which has the solutions $\lambda_1=Ai$, $\lambda_2=Bi$, where
\begin{equation}\label{eq1}
A=\frac{c+\sqrt{(2a-c)^2+4b^2}}{2},\qquad B=\frac{c-\sqrt{(2a-c)^2+4b^2}}{2}.
\end{equation}
We then find the general solutions for \eqref{eq9} as
\begin{equation} \label{eq111}
\begin{array}{ll}
g_1(t)=(\mu_1+i\mu_2)e^{iAt}+(\nu_1+i\nu_2)e^{iBt},\\
g_2(t)=(\eta_1+i\eta_2)e^{iAt}+(\xi_1+i\xi_2)e^{iBt},
 \end{array} \end{equation}
where $\mu_l,\nu_l,\eta_l,\xi_l\in \mathbb{R}, \, l=1,2.$
Since $A\neq B$, the functions $\sin(At)$, $\cos(At)$, $\sin(Bt)$, $\cos(Bt)$ are linearly independent. After substitution of (\ref{eq111}) in (\ref{eq8}) we get that the following hold
\begin{eqnarray}\label{eq11}
(A-a)\mu_1&=&b\xi_2,\qquad (B-a)\nu_1=b\eta_2, \nonumber\\
(A-a)\mu_2&=&b\xi_1,\qquad (B-a)\nu_2=b\eta_1, \nonumber\\
(A-a)\eta_1&=&-b\nu_2,\qquad (B-a)\xi_1=-b\mu_2, \nonumber\\
(A-a)\eta_2&=&-b\nu_1,\qquad (B-a)\xi_2=-b\mu_1. \nonumber
\end{eqnarray}
Since
$$
\frac{A-a}{b}=\frac{-b}{B-a}=\frac{c-2a+\sqrt{(c-2a)^2+4b^2}}{2b}:=\varphi,
$$
we get four relations among coefficients
\begin{equation}\label{eq12}
\varphi=\frac{\xi_2}{\mu_1}=\frac{\xi_1}{\mu_2}=-\frac{\nu_1}{\eta_2}=-\frac{\nu_2}{\eta_1}.
\end{equation}
Case $b=0$ leads to $c_2=0$, which has already been taken into consideration.

Hence, the solution for the function $f$ is
\begin{eqnarray}\label{eq13}
f(t)=(\mu_1\cos(At)-\mu_2\sin(At)-\varphi\eta_2\cos(Bt)+\varphi\eta_1\sin(Bt))+\\ \nonumber
(\mu_2\cos(At)+\mu_1\sin(At)-\varphi\eta_1\cos(Bt)-\varphi\eta_2\sin(Bt))i+\\ \nonumber
(\eta_1\cos(At)-\eta_2\sin(At)+\varphi\mu_2\cos(Bt)-\varphi\mu_1\sin(Bt))j-\\ \nonumber
(\eta_2\cos(At)+\eta_1\sin(At)+\varphi\mu_1\cos(Bt)+\varphi\mu_2\sin(Bt))k.
\end{eqnarray}
Since $f(t)$ is a curve on the sphere $\mathbb{S}^3$ through the point $1$, we get the following relations
\begin{equation}\label{eq14}
\mu_2=\eta_1=0,\quad \mu_1=\frac{1}{1+\varphi^2},\quad \eta_2=\frac{-\varphi}{1+\varphi^2}. \nonumber
\end{equation}
Finally, the solution of the equation is
$$f(t)=\bigg(\frac{1}{1+\varphi^2}\cos(At)+\frac{\varphi^2}{1+\varphi^2}\cos(Bt)\bigg)+
\left(\frac{1}{1+\varphi^2}\sin(At)+\frac{\varphi^2}{1+\varphi^2}\sin(Bt)\right)i+$$
$$\left(\frac{\varphi}{1+\varphi^2}\sin(At)-\frac{\varphi}{1+\varphi^2}\sin(Bt)\right)j-
\left(-\frac{\varphi}{1+\varphi^2}\cos(At)+\frac{\varphi}{1+\varphi^2}\cos(Bt)\right)k.$$
Let us notice that this is indeed a curve on $\mathbb{S}^3$ since $\mu_1^2+\mu_2^2+\eta_1^2+\eta_2^2=\frac{1}{1+\varphi^2}=\cos^2\theta$. We get the parametrization (\ref{eq15}) of the geodesics when we substitute (\ref{eq13}), with the proper constants $A,B,\tilde A,\tilde B$ into $\gamma(t)=(x(t),y(t))$.
\end{proof}

\begin{proposition}
\begin{enumerate}
\item \label{Cor1}
The geodesic lines on \nks\, with respect to nearly K\"ahler metric coincide with geodesic lines with respect to usual Euclidean product metric if and only if $c_1=0$ or $c_2=0$.

\item \label{Cor2}
   \begin{itemize}
      \item The tangent vector of the geodesic line is the eigenvector of the product structure $P$ with eigenvalue $-1$ if and only if $c_1=0$.
      \item The tangent vector of the geodesic line is the eigenvector of the product structure $P$ with eigenvalue $1$ if and only if $c_1\neq 0$, $c_2=0$, $d_1=0$.
    \end{itemize}
\item \label{Cor3}
The geodesic line on \nks is closed if and only if it has the parametrization from the case Theorem \normalfont{\ref{Geodlinesth}} {\normalfont{(1)}}, or the fractions $\frac{a}{\tilde a}$, $\frac{B}{A}$ and $\frac{\tilde B}{\tilde A}$ are rational numbers in the cases {\normalfont{(2)}} and {\normalfont{(3)}} of Theorem \normalfont{\ref{Geodlinesth}}, respectively.
\end{enumerate}
\end{proposition}

\begin{proof}
\begin{enumerate}
\item
From relation (\ref{x}) it is evident that $\nabla^E_{\gamma'}\gamma'=0$ is equivalent to $\alpha'=\beta'=0$, namely $\alpha=\beta=\mathrm{const}$. We have already seen that this is true if and only if $c_1=0$ or $c_2=0$.

\item
As $\gamma'=(x\alpha,y\beta)$, $P\gamma'=(x\beta,y\alpha)$, condition $P\gamma'=\pm\gamma'$ reduces to $\alpha=\pm\beta$ and the conclusion follows from (\ref{c1}) and (\ref{c2}).

\item
It is apparent that $\gamma(t)$ is a periodic function in the case Theorem \ref{Geodlinesth} (1). Also, in the case Theorem \ref{Geodlinesth} (2) both coordinate functions of $\gamma(t)$ are periodic with the same period if and only if the fraction $\frac{a}{\tilde a}$ is a rational number (if one of the constants $a,\tilde a$ is $0$, then this condition stands for the second one which is nonzero). In the case Theorem \ref{Geodlinesth} (3) the first coordinate function of $\gamma(t)$ is periodic if and only if the functions $\sin(At)$, $\cos(At)$, $\sin(Bt)$, $\cos(Bt)$ are periodic with the same period, which is again true if and only if the fraction $\frac{B}{A}$ is a rational number (if one of the constants $A$, $B$ is $0$, then this condition stands for the second one which is nonzero). A similar conclusion is valid for the second coordinate function.
\end{enumerate}
\end{proof}

\begin{remark} Using the formula (\ref{g}) we can calculate the length of the tangent vector $$\|\gamma'\|=\sqrt{g(\gamma',\gamma')}=\sqrt{(\frac13+d_1^2)\|c_1\|^2+\|c_2\|^2}.$$ Since this is constant, it is easy to get the arclength parametrization where $\gamma'$ is a unit length vector field along the geodesics, only with the change of parameter $s=t\sqrt{(\frac13+d_1^2)\|c_1\|^2+\|c_2\|^2}$.
\end{remark}

\end{document}